\documentclass[oneeqnum]{siamltex}

\usepackage{mathptmx} 
\usepackage{array}
\usepackage{graphicx}
\usepackage{amssymb}
\usepackage{epstopdf}
\usepackage{amsmath}
\usepackage{amsfonts}
\usepackage{multirow}
\usepackage{verbatim}
\usepackage{color}

\newcommand*\samethanks[1][\value{footnote}]{\footnotemark[#1]}

\newtheorem{assumption}[theorem]{Assumption}

\def\bu{\mathbf{u}}
\def\bp{\mathbf{p}}

\def\dx{\dot{x}}

\def\dq{\dot{q}}

\def\Tt{^{{\mbox{\tiny \bf \sf T}}}}

\def\J{{\cal{J}}}
\def\K{{\cal{K}}}

\def\bx{{\mathbf{x}}}
\def\bq{{\mathbf{q}}}
\def\be{{\mathbf{e}}}

\def\tP{{\tilde{P}}}
\def\tU{{\tilde{U}}}

\def\tbu{{\tilde{\bu}}}

\def\hc{{\hat{c}}}
\def\hb{{\hat{b}}}

\def\hx{{\hat{x}}}
\def\hq{{\hat{q}}}

\def\hbeta{{\hat{\beta}}}
\def\hgamma{{\hat{\gamma}}}

\def\setZ{\mathbb{Z}}
\def\setU{\mathbb{U}}
\def\setP{\mathbb{P}}
\def\setR{\mathbb{R}}

\begin{document}
\title{Symmetric Strong Duality for a Class of Continuous Linear Programs with Constant Coefficients}

\author{Evgeny Shindin \thanks{
Department of Statistics,
The University of Haifa,
Mount Carmel 31905, Israel; 
{\tt shindin@netvision.net.il, gweiss@stat.haifa.ac.il} 
Research supported in part by
Israel Science Foundation Grant and 711/09.}          
\and Gideon Weiss \samethanks}
\maketitle

\begin{abstract}
We consider Continuous Linear Programs over a continuous finite time horizon $T$, with linear cost coefficient functions and linear right hand side functions and a constant coefficient matrix, where we search for optimal solutions in the space of measures or of functions of bounded variation.  
These models generalize the Separated Continuous Linear Programming models and their various duals, as formulated in the past by Anderson, by Pullan, and by Weiss.  
We present simple necessary and sufficient conditions for feasibility.  
We formulate a symmetric dual and investigate strong duality by considering discrete time approximations.  We prove that under a Slater type condition there is no duality gap and there exist optimal solutions which have impulse controls at $0$ and $T$ and have piecewise constant densities in $(0,T)$.  Moreover, we show that under non-degeneracy assumptions all optimal solutions are of this form, and are uniquely determined over $(0,T)$.
\end{abstract}
\medskip


\pagestyle{myheadings}
\thispagestyle{plain}
\markboth{E. SHINDIN AND G. WEISS}{SYMMETRIC STRONG DUALITY FOR A CLASS OF CLP}

\pagenumbering{arabic}

\section{Introduction}
\label{sec.introduction}
We consider problems of the form:
\begin{eqnarray}
\label{eqn.mpclp}
&\max & \int_{0-}^T (\gamma+ (T-t)c)\Tt dU(t)  \nonumber\\
\mbox{M-CLP}\, & \mbox{s.t.} & \qquad A\, U(t)  \quad \le \beta + b t, 
\quad 0 \le t \le T, \\
 &&  U(t) \ge 0,\; U(t) \mbox{ non-decreasing and right continuous on [0,T].}\nonumber 
\end{eqnarray}
where $A$ is a $K\times J$ constant matrix, $\beta,b,\gamma,c$ are constant vectors of corresponding dimensions, the integrals are Lebesgue-Stieltjes,  $U$ are $J$ unknown functions over the time horizon $[0,T]$, and by convention we take $U(0-)=0$.

We formulate a symmetric dual problem
\begin{eqnarray}
\label{eqn.mdclp}
&\min & \int_{0-}^T (\beta+(T-t)b)\Tt dP(t)  \nonumber \\
\mbox{M-CLP$^*$}\, & \mbox{s.t.} & \qquad  A\Tt P(t)  \quad \ge \gamma + c t,  
\quad 0 \le t  \le T, \\
 &&   P(t) \ge 0,\; P(t) \mbox{ non-decreasing and right continuous on [0,T].}\nonumber 
\end{eqnarray}
with $K$ unknown dual functions $P$ with the same convention $P(0-)=0$.  It is convenient to think of dual time as running backwards, so that $P(T-t)$ corresponds to $U(t)$.

 The main feature to note here is that the objective as well as the left hand side of the constraints are formulated as Lebesgue-Stieltjes integrals with respect to a vector of  monotone non-decreasing control function $U(t)$, in other words our controls are in the space of measures.  This is in contrast to most formulations in which the objective and left hand side of the constraints are Lebesgue integrals with respect to a measurable bounded control $u(t)$, in other words  controls which are in the space of densities.  In particular, while in the usual formulation the left hand side of the constraints is an absolutely continuous function, our formulation allows the left hand side of the constraint to have jumps, as a result of jumps in $U(t)$, which correspond to impulse controls.

Our main results in this paper include the following:
\begin{itemize}
\item
We discuss how this formulation relates to and generalizes  previous continuous linear programs.
\item
We show weak duality and present a simple necessary and sufficient test for feasibility of M-CLP. We also present a  Slater type condition which is easily checked, using the same test.
\item
We show that under this Slater type condition there is no duality gap between M-CLP and M-CLP$^*$, by considering discrete time approximations. We also show that in this case M-CLP and M-CLP$^*$ posses optimal solutions.
\item
We further show that in that case there exist optimal solutions for which
  $U(t)$ and $P(t)$ have impulse controls at $0,T$ and  are absolutely continuous inside $(0,T)$, with piecewise constant densities. 
 \item 
 Finally, under appropriate simple non-degeneracy assumptions we show that  all optimal solutions are of this form, and that the absolutely continuous part on $(0,T)$ is uniquely determined.
\end{itemize}  
Further research to develop a simplex-type algorithm that constructs solutions of this form is in progress.

We note that the question of existence of strong duality, and whether symmetric dual formulations are useful is far from simple when dealing with linear programs in infinite dimensional spaces \cite{barvinok:02,shapiro:01}.  Our results in this paper furnish an example where indeed strong duality can hold with a symmetric dual, if a Slater type condition is satisfied.

\section{Background and motivation}

Continuous linear programs were introduced by Bellman in 1953 \cite{bellman:53,bellman:57} to model economic processes:  find a bounded measurable $u$ which
\begin{eqnarray}
\label{eqn.bellman}
&\max & \int_0^T c\Tt (t) u(t) dt \nonumber \\
\mbox{Bellman-CLP  }\quad \quad &\mbox{s.t.} & H(t) u(t) + \int_0^t G(s,t) u(s) ds \le
a(t),  \hspace{0.6in}
\\ &&  u(t) \ge 0, \quad t \in [0,T].  \nonumber 
\end{eqnarray}
Where $G(s,t),H(t)$ are given matrix functions.
These problems were investigated by Dantzig and some of his students, to model continuous time Leontief systems, and by several other early authors \cite{dantzig:51,grinold:70,levinson:66,tyndall:65,tyndall:67}, with many publications since, but up to date no efficient algorithms or coherent theory have emerged, and these problems are considered very hard.

Separated continuous linear programs (SCLP) were introduced by Anderson \cite{anderson:78,anderson:81} in the context of job-shop scheduling:
 \begin{eqnarray}
\label{eqn.ASCLP}
&\max & \int_0^T c(t)\Tt u(t) \, dt \nonumber\\
\mbox{Anderson-SCLP    }\quad &\mbox{s.t.} & \int_0^t G u(s) ds  \le
 a(t), \hspace{1.7in}\\
&&  \quad  H u(t)  \qquad \le b(t),
\nonumber
\\ &&  u(t) \ge 0, \quad t \in [0,T]  \nonumber.
\end{eqnarray}
where $G,H$ are constant matrices, and $a(t),b(t),c(t)$ are given  vector functions.  
Some special cases of SCLP were solved by Anderson and Philpott \cite{anderson-philpott:89,anderson-philpott:89:2}, and  this research and related earlier work were summarized in the 1987 book of Anderson and Nash \cite{anderson-nash:87}, which also contains many references to work on CLP up to that date.

Major progress in the theory of SCLP was achieved by Pullan \cite{anderson-pullan:96}, \cite{pullan:93}--\cite{pullan:02}.   Pullan
considered SCLP problems with $a(t)$, $ b(t)$ and $c(t)$ piecewise analytic, and formulated a non-symmetric dual to (\ref{eqn.ASCLP}) (here we modify Pullan's original version by  letting the dual run in reversed time, as in (\ref{eqn.mdclp})):
\begin{eqnarray}
\label{eqn.pullandsclp}
&\min & \int_0^T a(T-t)\Tt d P(t) + \nonumber \int_0^T b(T-t)\Tt q(t) d t  \hspace{0.6in}\\
\mbox{Pullan-SCLP$^*$   }\quad 
&\mbox{s.t.}&   G\Tt  P(t) + H\Tt q(t) \ge c(T-t),  \\
&&  P(t)\ge 0,\; P(t) \mbox{ non-decreasing and right continuous on [0,T].}  \nonumber \\	  
&&  q(t) \ge 0, \; t \in [0,T],\;   \nonumber
\end{eqnarray}
Pullan showed that when the feasible region of $H u(t) \le b(t)$ is bounded strong duality holds between (\ref{eqn.ASCLP}) and (\ref{eqn.pullandsclp}). In the special case that $a(t),c(t)$ are piecewise linear and $b(t)$ piecewise constant Pullan provided an infinite but convergent algorithm to solve the problems and observed that $P$ was absolutely continuous, except for atoms at the breakpoints of $a,b,c$.

The results of Pullan raised several questions:
\begin{itemize}
\item
Is the boundedness restriction necessary?
\item
Can one formulate a symmetric dual?
\item
Do solutions of the form observed by Pullan always exist? 
\end{itemize}
More recently Weiss \cite{weiss:08} considered the following SCLP problem
\begin{eqnarray}
\label{eqn.PWSCLP}
&\max & \int_0^T (\gamma+ (T-t)c)\Tt u(t) + d\Tt x(t) \,dt   \nonumber   \hspace{1.0in} \\
\mbox{SCLP}  & \mbox{s.t.} &  \int_0^t G\, u(s)\,ds  + F x(t) \le \alpha + a t \\
 && \quad\; H u(t) \le b \nonumber \\
&& \quad x(t), u(t)\ge 0, \quad 0\le t \le T.  \nonumber
\end{eqnarray}
and the symmetric dual
\begin{eqnarray}
\label{eqn.DWSCLP}
&\min & \int_0^T (\alpha + (T-t)a)\Tt p(t) + b\Tt q(t) \,dt     \nonumber  \hspace{0.9in} \\
\mbox{SCLP$^*$}  &\mbox{s.t.} &  \int_0^t  G\Tt\, p(s)\,ds + H\Tt q(t) \ge \gamma + c t \\
 && \quad\; F\Tt p(t) \ge d \nonumber \\
&& \quad q(t), p(t)\ge 0, \quad 0\le t \le T.   \nonumber
\end{eqnarray}
with constant vectors and matrices $G,F,H,\alpha,a,b,\gamma,c,d$.
In contrast to previous work Weiss developed a simplex type algorithm which solves this pair of problems exactly, in a finite bounded number of steps, without using discretization.

The simplex type algorithm of Weiss can solve any pair of problems (\ref{eqn.PWSCLP}), (\ref{eqn.DWSCLP}) which possess optimal solutions $u(t),p(t)$ that are bounded measurable functions. It produces solutions with $u(t),p(t)$  piecewise constant, and $x(t),q(t)$  continuous piecewise linear.  However, there exist problems for which both (\ref{eqn.PWSCLP}) and (\ref{eqn.DWSCLP}) are feasible but either (\ref{eqn.PWSCLP}) or (\ref{eqn.DWSCLP}) or both do not possess optimal solutions $u(t),p(t)$ in the space of bounded measurable functions. Moreover, one can construct examples, where (\ref{eqn.PWSCLP}) possess optimal solutions in the space of bounded measurable functions, but (\ref{eqn.DWSCLP}) is infeasible.
Such problems cannot be solved by the algorithm of Weiss.   
This raises the question whether they can be solved in the space of measures, and motivates our formulation of M-CLP, M-CLP$^*$ problems (\ref{eqn.mpclp}), (\ref{eqn.mdclp}).

\begin{definition}
Consider the SCLP problem (\ref{eqn.PWSCLP}). Then the M-CLP problem with the following data:
\begin{equation*}
\begin{array}{c}
A = \left[ \begin{array}{cccc} G & 0 & F & -F \\   0 & 0 & -I & I \\  
H & I & 0 & 0  \\  -H & -I & 0 & 0 \end{array}  \right] \quad
U(t) = \left[  \begin{array}{c} U_*(t) \\  U_s(t) \\ U^{+}(t) \\ U^{-}(t)  \end{array}  \right] \quad
\beta^* + b^*t =  \left[  \begin{array}{c} \alpha \\  0 \\ 0 \\ 0  \end{array}  \right] +
 \left[ \begin{array}{c} a \\  0 \\ b \\ -b  \end{array} \right] t, \\
\gamma^* + (T-t)c^* = \left[ \gamma \quad 0 \quad\;  d  \quad -d \right] + (T-t)\left[ c \quad 0 \quad\;  0  \quad 0 \right]
\end{array} 
\end{equation*}
is called the M-CLP extension of SCLP.
\end{definition}
\begin{theorem}
\label{thm.generalization}
M-CLP/M-CLP$^*$ are generalizations of  SCLP/SCLP$^*$ in the  following sense: \\
(i) if  SCLP (\ref{eqn.PWSCLP}) and SCLP$^*$ (\ref{eqn.DWSCLP}) possess optimal solutions, then these solutions determine optimal solutions of the corresponding M-CLP/M-CLP$^*$ extensions with the same objective value. \\
(ii) If the M-CLP/M-CLP$^*$  extensions of the SCLP/SCLP$^*$ have  optimal solutions with no duality gap which are absolutely continuous, then this solution determines  optimal solutions of the 
SCLP/SCLP$^*$,  with the same objective value.\\
(iii) If SCLP is feasible and the Slater type condition \ref{def.slater} holds for M-CLP/M-CLP$^*$ extensions, then the supremum of the objective of SCLP is equal to the objective value of the optimal solution of the M-CLP extension. 
%
\end{theorem} 
\begin{proof}
(i) Consider an optimal solution $x^*(t), u^*(t)$ of (\ref{eqn.PWSCLP}).  By the Structure Theorem (Theorem 3 in \cite{weiss:08}) $x^*(t)$ is absolutely continuous and hence of bounded variation.  Therefore we can write $x^*(t)$ as the difference of two non-decreasing functions $x^*(t) = U^{+}(t) - U^{-}(t)$.
Let $u_s(t)$ be the slacks of the constraints $Hu(t) \le b$, and let $U(t) = \int_0^t u^*(t) dt$, $U_s(t) = \int_0^t u_s(t) dt$.  Then the resulting $\tU = [U_*, U_s,U^{+},U^{-}]$    satisfies the constraints of the M-CLP extension, with the same objective value.  

A similar argument applies to an optimal solution  $q^*(t), p^*(t)$ of (\ref{eqn.DWSCLP}), which 
determines a feasible solution $\tP$ of the  M-CLP$^*$ extension, which is dual to  M-CLP, and has the same objective values.  

Weak duality of 
M-CLP and M-CLP$^*$ (see Proposition \ref{thm:weakduality} below) then shows that these solutions are the optimal solutions of M-CLP and M-CLP$^*$. 

(ii)  If the solution of the M-CLP extension is absolutely continuous then taking $u(t)=\frac{dU(t)}{dt}$ and $x(t)=U^+(t)-U^-(t)$ we get a feasible solution of SCLP, with the same objective value.  
The same holds for SCLP$^*$, and by weak duality these are optimal solutions.

(iii) The proof of this part is postponed to Section \ref{sec.solform}, after Theorem \ref{thm.solform}.
\qquad\end{proof}

It is not hard to see that  (\ref{eqn.mpclp}), (\ref{eqn.mdclp}) generalize also
Anderson and Pullan's problems (\ref{eqn.ASCLP}), (\ref{eqn.pullandsclp}) restricted to $a(t),c(t)$ affine, and  $b(t)$  constant.


\section{Weak duality, complementary slackness and feasibility}
\hspace*{1 cm}

\begin{proposition}
\label{thm:weakduality}
Weak duality holds for M-CLP, M-CLP$^*$ (\ref{eqn.mpclp}),(\ref{eqn.mdclp}).
\end{proposition}
\begin{proof}
Let $U(t),\,P(t)$ be feasible solutions for  (\ref{eqn.mpclp}), (\ref{eqn.mdclp}), and compare
their objective values:
\begin{eqnarray}
\label{eqn.weakduality}
&& \mbox{Dual objective }= \nonumber \\
&=&  \int_0^T \big(\beta+(T-t)b \big)\Tt  dP(t)  
\nonumber \\
&\ge&   
 \int_0^T  \Big( \int_0^{T-t} A dU(s)  \Big)\Tt dP(t)  \nonumber \\
&=&   \int_0^T  \Big( \int_0^{T-t} A\Tt dP(s) \Big)\Tt dU(t)  \nonumber \\
&\ge&  \int_0^T  \big( \gamma + (T-t) c \big)\Tt dU(t)    \nonumber \hspace{0.8in} \\
&& =  \mbox{Primal objective. }  \nonumber
\end{eqnarray}
The first inequality follows from the primal constraints at $T-t$, and from $P(t)$ non-decreas\-ing.  The equality follows by changing order of integration, using Fubini's theorem.  The second inequality follows from the dual constraints at $T-t$, and from $U(t)$ non-decreasing.
\qquad\end{proof}

Equality of the primal (M-CLP) and dual (M-CLP$^*$) objective will occur if and only if the following holds:

{\bf Complementary slackness condition.}
Let $x(t)=\beta+bt-AU(t)$ and $q(t)= A\Tt P(t) - \gamma-ct $ be the slacks in (\ref{eqn.mpclp}), (\ref{eqn.mdclp}).  The complementary slackness condition for M-CLP, M-CLP$^*$ is
\begin{equation}
\label{eqn.gcompslack}
 \int_0^T x(T-t)\Tt dP(t) =  \int_0^T q(T-t)\Tt dU(t) = 0.
\end{equation}

In the following propositions and theorems in this and following sections we present results for M-CLP. By  symmetry these results hold for M-CLP$^*$, with the obvious modifications.

We present now a simple necessary and sufficient condition for feasibility.  This is   similar to a condition derived by Wang, Zhang and Yao \cite{wang-zang-yao:09}.  It involves the standard linear program Test-LP and its dual Test-LP$^*$.
 \begin{eqnarray}
\label{eqn.ptestLP}
 &\max &z = (\gamma + c T)\Tt   \bu + \gamma\Tt U  \nonumber  \hspace{1.0in} \\
 \mbox{Test-LP} &\mbox{s.t.} &  A  \bu  \le \beta  \\
 &\mbox{} &  A  \bu + A U  \le \beta + bT \nonumber \\
&& \quad \bu,\, U  \ge 0  \nonumber 
\end{eqnarray}
 \begin{theorem}
\label{the.feas}
M-CLP is feasible if and only if Test-LP is feasible. 
  \end{theorem}
\begin{proof}
(i) \textit{Sufficiency:} Let $\bu,\, U$   be a solution of Test-LP  (\ref{eqn.ptestLP}), with slacks $x^0=\beta-A\bu$, $x^T=\beta+bT -A\bu - A U$. 
Then  $U(t)=\bu + \frac{t}{T}U,\,0\le t \le T$ is a feasible solution of M-CLP (\ref{eqn.mpclp}), with non-negative slacks $x(t)= (1 - \frac{t}{T}) x^0 +  \frac{t}{T} x^T$.
  To check this we have for $0 \le t \le T$:
\begin{eqnarray*}
A U(t) + x(t) &=&  A \left(\bu + \frac{t}{T}U\right) + \left(1 - \frac{t}{T}\right) x^0 +  \frac{t}{T} x^T \\
&=&  \frac{t}{T} \left(A \bu + A U + x^T\right)  + \left(1 - \frac{t}{T}\right) \left(A \bu + x^0\right) \\
&=&  \beta + b t.     
\end{eqnarray*}
(ii) \textit{Necessity:} Let $U(t)$ be a feasible solution of M-CLP (\ref{eqn.mpclp}) with slack $x(t)\ge0$. Then $\bu=U(0),\; U = \int_{0^+}^T  dU(t)$ with slack $x^0=x(0), x^T=x(T)$ is a feasible solution for Test-LP (\ref{eqn.ptestLP}), as is seen immediately.  
\qquad\end{proof}

We use the following definition:
\begin{definition}[\bf Slater type condition]
\label{def.slater}
We say that the Test-LP problem (\ref{eqn.ptestLP}) is strictly feasible at $T$ if there exists a feasible solution  $\bu, U$ of (\ref{eqn.ptestLP}) and a constant $\alpha>0$ such that $\beta - A\bu \ge \alpha$ and $\beta + bT - A\bu - AU \ge \alpha$. We say that M-CLP is strictly feasible at $T$ if there exists a feasible solution  $U(t)$ of (\ref{eqn.mpclp}) and a constant $\alpha>0$ such that $\beta + bt - A U(t) \ge \alpha$ for all $t \in [0, T]$.
\end{definition}
\begin{corollary}
\label{cor.strfeas}
M-CLP is strictly feasible if and only if Test-LP is strictly feasible.
\end{corollary}
\begin{proof}
Simply define $\beta^* = \beta - \alpha$ and recall Theorem \ref{the.feas} for problems with $\beta$ replaced by $\beta^*$.
\qquad\end{proof}


\section{Discrete time approximations and strong duality}
\label{sec.discretizations}
In this section we consider a pair of M-CLP/M-CLP$^*$ problems which are feasible, and  use time discretization to solve them approximately.  We prove that if  M-CLP and M-CLP$^*$ are strictly feasible, then there is no duality gap and an optimal solution exists.  We use a discretization approach similar to \cite{pullan:93}.


\subsection{General discretizations}

For a partition  $0=t_0 < t_1 < \ldots < t_N=T$ we define  the discretization of M-CLP to be:
\begin{eqnarray}
\label{eqn.dCLP1}
 &\max &z=(\gamma + c T)\Tt  \bu^0 + \sum_{i=1}^N \left(\gamma+ \left(T - \frac{t_i + t_{i-1}}{2}\right) c \right)\Tt \left(t_i - t_{i-1}\right) u^i  + \gamma\Tt \bu^N    \nonumber   \\
  &\mbox{s.t.} &  \displaystyle A \bu^0 + x^0 = \beta   \nonumber \\
 \mbox{dCLP$_1$} &\mbox{} & \displaystyle A \bu^0 + A \sum_{i=1}^n \left(t_i - t_{i-1}\right) u^i + x^n = \beta + b t_n \quad \text{for } n=1,\dots,N    \\
 &\mbox{} &  \displaystyle A \bu^0 + A \sum_{i=1}^N \left(t_i - t_{i-1}\right) u^i + A \bu^N + \bx^N  = \beta + b T   \nonumber \\
&& \quad \bu^0,\, u^1,\dots,u^N, \bu^N, \,x^0, x^1,\dots,x^N, \bx^N \ge 0.  \nonumber
\end{eqnarray}
and for the same time partition the discretization of M-CLP$^*$ is defined as:
\begin{eqnarray}
\label{eqn.dCLP2}
 &\min &z=(\beta\Tt +Tb\Tt) \bp^N  + \sum_{i=1}^N \left(\beta\Tt +  \frac{t_{i} + t_{i-1}}{2} b\Tt\right) \left(t_i - t_{i-1}\right) p^i  +  \beta\Tt  \bp^0   \nonumber  \\
 &\mbox{s.t.} & \displaystyle A\Tt \bp^N - q^N = \gamma     \\
 \mbox{dCLP$_2$}   &\mbox{} & \displaystyle A\Tt \bp^N +  A\Tt \sum_{i=n}^N \left(t_i - t_{i-1}\right) p^i 
   - q^{n-1} = \gamma + c (T - t_{n-1}) \quad \text{for } n=N,\dots,1 \nonumber  \\
 &\mbox{} &    \displaystyle A\Tt \bp^N  + A\Tt \sum_{i=1}^N \left(t_i - t_{i-1}\right) p^i + A\Tt \bp^0 - \bq^0  = \gamma + c T  \nonumber \\
 && \quad \bp^N,\, p^N,\dots,p^1, \bp^0, \,q^N, \dots,q^1, q^0, \bq^0 \ge 0.  \nonumber
\end{eqnarray}
Note that these two problems are not dual to each other.

Following Pullan \cite{pullan:93}, for a partition $0=t_0 < t_1 < \ldots < t_N=T$ and values  $f(t_0),\dots,$ $f(t_N)$ we define the \textit{piecewise linear extension}:
\[
f_L(t) = \left( \frac{t_i - t}{t_i - t_{i-1}} \right) f (t_{i-1}) +  \left( \frac{t - t_{i-1}}{t_i - t_{i-1}}\right) f (t_{i}) \quad \text{for } t \in [t_{i-1}, t_i)  \text{ for } i = 1,\dots,N.
\]
and the \textit{piecewise constant extension}:
\[
f_C(t)= f(t_i), \quad  t \in [ t_{i-1}, t_i ),\;i=1,\ldots N.
\]

The following proposition is an easy extension of Theorem \ref{the.feas}.
\begin{proposition}
\label{thm.discrete}
All discretizations dCLP$_1$ (\ref{eqn.dCLP1}) are feasible if and only if M-CLP is feasible.
\end{proposition}
\begin{proof}
(i) Let $U(t)$ be a feasible solution of M-CLP (\ref{eqn.mpclp}) with slacks $x(t)\ge 0$. Then $\bu^0=U(0),\; u^n = \frac{1}{t_n -t_{n-1}}\int_{t_{n-1}}^{t_n}  dU(t)$ 
(in these integrals we take $t_0=0$ and $t_n = t_n-$ for $n=1,\dots,N$),  $\bu^N=U(T)-U(T-)$, 
and $x^0=x(0), \; x^n=x(t_n-),\,n=1,\ldots,N,\;  \bx^N=x(T)$ is a feasible solution for dCLP$_1$ (\ref{eqn.dCLP1}). To check this we have for $n=0,\ldots,N$:
\begin{eqnarray*} 
&& A \bu^0 + A \sum_{i=1}^n \left(t_i - t_{i-1}\right) u^i + x^n  =  \\
&& A  U(0) + A \sum_{i=1}^n \left(t_i - t_{i-1}\right) \int_{t_{i-1}}^{t_i}\frac{1}{t_i -t_{i-1}} dU(t) +x(t_n-)  = A U(t_n-) + x(t_n-) =\beta + bt_n
\end{eqnarray*}

(ii) Let $\bu^0,\, u^1,\dots,u^N, \bu^N$ be a feasible solution of dCLP$_1$.  Define $U(0)=\bu^0$, let $u(t)$ be the piecewise constant extension of $u_1,\ldots,u_N$, and let $U(t)=U(0)+\int_0^t u(s) ds,\,t\in [0,T)$,  $U(T)=U(T-)+\bu^N$.  Then $U(t)$ is a feasible solution of M-CLP.
\qquad\end{proof}
\begin{proposition}
\label{thm.feas-ext}
Any feasible solution of dCLP$_1$ can be extended to a feasible solution of M-CLP with equal objective value.
\end{proposition}
\begin{proof}
We set $u(t)$ to be the piecewise constant extension of $u^1,\dots,u^N$ and take $U(t)$ to be the measure with density $u(t)$ on $(0,T)$ and impulses $U(\{0\}) = \bu^0, U(\{T\})=\bu^N$. We also set $x(t)$ to be the piecewise linear extension of $x^0,\dots,x^N$, and take $\bx^N$ to be the same for both problems. It is immediate to see that this gives a feasible solution to M-CLP.  Furthermore, it is immediate to see that the objective of dCLP$_1$ equals the objective of M-CLP for this extended solution. 
\qquad\end{proof}

\begin{proposition}
\label{thm.dclp-opt}
 The optimal values $V$ of the various problems satisfy:
\[
V(dCLP_1) \le V(M\text{-}CLP) \le V(M\text{-}CLP^*) \le V(dCLP_2) 
\] 
\end{proposition}
\begin{proof}
The first and last inequalities follow from Proposition \ref{thm.feas-ext}  and the middle inequality follows from weak duality.
\qquad\end{proof}


\subsection{Discretizations with equidistant partitions}

Similar to Wang, Zhang and Yao \cite{wang-zang-yao:09} and to Pullan \cite{pullan:93}  we use \textit{even equidistant} partitions, denoted $\pi^N$ which divides the interval $[0, T]$ into $N$ equal segments, each of length $2\epsilon$, i.e. $\epsilon = \frac{T}{2N}$. 
With this partition we introduce the  notations:
\begin{itemize}
\item
Given a $K \times J$ matrix $A$ we define the $N K \times  J$ matrix $A_\|$  ,
the $N K \times N J$ matrix $A_\blacktriangle$, 
 and the $ K \times NJ$ matrix $A_=$ as follows:
\[
A_\blacktriangle=\left[ \begin{array}{cccc} 
A &  &  &  \\
A & A & &  \\
\dots & & & \\
A & A & \dots & A \\
\end{array} \right], \quad  
A_\|=\left[ \begin{array}{c} 
A   \\
A   \\
\dots \\
A  \\
\end{array} \right],
\quad
A_= = \left[ A \quad A \quad \dots \quad A \right].
\]
\item
We define the $N$-fold vectors, each  with $N$ vector components:
\end{itemize}
\[
\hbeta =  \left[ \begin{array}{c} 
\beta \\
 \vdots \\
 \beta 
 \end{array}\right]
 , \quad
\hgamma =  \left[ \begin{array}{c} 
\gamma \\
 \vdots \\
 \gamma
 \end{array}\right]
\]
\[
\hb_1=\left[ \begin{array}{c}
2b\epsilon \\
4b\epsilon \\
\dots \\
b T
\end{array}\right] \quad \hb_2=\left[ \begin{array}{c}
b\epsilon \\
3b\epsilon \\
\dots \\
b (T -\epsilon)
\end{array}\right]  \quad \hc_1=\left[ \begin{array}{c}
c (T -\epsilon) \\
c (T -3\epsilon)  \\
\dots \\
c\epsilon
 \end{array}\right]  \quad \hc_2=\left[ \begin{array}{c}
c T  \\
c (T- 2\epsilon) \\
\dots \\
2 c\epsilon
\end{array}\right] 
\]
\[
\Delta U = \left[ \begin{array}{c}
 \Delta U^1 \\
 \vdots \\
\Delta U^N 
 \end{array}\right] = \left[ \begin{array}{c}
2 u^1 \epsilon \\
 \vdots \\
2 u^N \epsilon
 \end{array}\right], 
 \Delta P  = \left[ \begin{array}{c}
 \Delta P^1 \\
 \vdots \\
\Delta P^N 
 \end{array}\right] = \left[ \begin{array}{c}
2 p^1 \epsilon \\
 \vdots \\
2 p^N \epsilon
 \end{array}\right], 
 \hx = \left[ \begin{array}{c}
x^1  \\ \vdots \\  x^N \end{array}\right],
 \hq  = \left[ \begin{array}{c}
q^0  \\ \vdots \\ q^{N-1}  \end{array}\right]
\]

Using this notation we rewrite problems (\ref{eqn.dCLP1}), (\ref{eqn.dCLP2}) for even equidistant partitions, as:
\begin{eqnarray}
\label{eqn.dCLPpi1}
 &\max &z =\displaystyle (\gamma + c T)\Tt  \bu^0 + \left(\hgamma +\hc_1\right)\Tt \Delta U + \gamma\Tt \bu^N    \nonumber  \hspace{0.6in} \\
  &\mbox{s.t.} &  \displaystyle A \bu^0 + x^0 = \beta   \nonumber  \\
 \mbox{dCLP$_1(\pi^N)$} &\mbox{} & \displaystyle A  _\| \bu^0 + A_\blacktriangle \Delta U + \hx = \hbeta + \hb_1  \\
 &\mbox{} &  \displaystyle A \bu^0 + A_= \Delta U+ A \bu^N + \bx^N  = \beta + b T  \nonumber \\
&& \quad \bu^0,\, \Delta U, \bu^N, \,x^0,\hx, \bx^N \ge 0.  \nonumber
\end{eqnarray}
\begin{eqnarray}
\label{eqn.dCLPpi2}
 &\min &z =\displaystyle (\beta + b T)\Tt  \bp^N + \left(\hbeta + \hb_2 \right)\Tt\Delta P  + \beta\Tt \bp^0      \nonumber  \hspace{0.6in}  \\
 &\mbox{s.t.} &  \displaystyle  A\Tt \bp^N - q^N = \gamma  \nonumber  \\
 \mbox{dCLP$_2(\pi^N)$}  &\mbox{} & \displaystyle  A\Tt  _\| \bp^N + A_\blacktriangle\Tt \Delta P  - \hq = \hgamma + \hc_2  \\
 &\mbox{} &   \displaystyle  A\Tt \bp^N +  A\Tt _= \Delta P + A\Tt \bp^0 - \bq^0  = \gamma + c T \nonumber  \\
&& \quad \bp^N,\, \Delta P, \bp^0, \,q^N,  \hq, \bq^0 \ge 0.  \nonumber
\end{eqnarray}
The reader may notice that in (\ref{eqn.dCLPpi2}) we have for convenience reversed the order of variables and the order of the constraints in the middle part of the problem relative to  (\ref{eqn.dCLP2})

To quantify the discretization error for time partition $\pi^N$ we define a modified pair of problems mdCLP$(\pi^N)$, mdCLP$^*(\pi^N)$:  
\begin{eqnarray} \mbox{mdCLP$(\pi^N)$}  & \max & z = (\gamma + c T)\Tt  \bu^0 + \left(\hgamma +\hc_2 \right)\Tt \Delta U + \gamma\Tt \bu^N \nonumber \\
&\mbox{s.t.} & \text{Constraints of (\ref{eqn.dCLPpi1})} \nonumber
\end{eqnarray}
\begin{eqnarray} \mbox{mdCLP$^*(\pi^N)$} & \min & z = (\beta + b T)\Tt  \bp^N + \left(\hbeta + \hb_1 \right)\Tt\Delta P  + \beta\Tt \bp^0 \nonumber \\
& \mbox{s.t.} & \text{Constraints of (\ref{eqn.dCLPpi2})} \nonumber
\end{eqnarray}
We note that they are dual to each other. They are both feasible if (\ref{eqn.dCLPpi1}), (\ref{eqn.dCLPpi2}) are feasible.
Moreover, since (\ref{eqn.dCLPpi1}), (\ref{eqn.dCLPpi2}) are (\ref{eqn.dCLP1}), (\ref{eqn.dCLP2}) rewritten, then problems mdCLP$(\pi^N)$ and mdCLP$^*(\pi^N)$ are feasible if and only if M-CLP, M-CLP$^*$ are feasible, by Proposition \ref{thm.discrete}. In this case  mdCLP$(\pi^N)$ and mdCLP$^*(\pi^N)$  also posses optimal solutions.  Denote by $\bu^{0*},  \Delta U^*,$ $ \bu^{N*}$ and $\bp^{N*},  \Delta P^*, \bp^{0*}$ an optimal solution of mdCLP$(\pi^N)$  and mdCLP$^*(\pi^N)$.
 
\begin{proposition}
\label{the.gapbound}
If M-CLP and M-CLP$^*$ are feasible then by solving mdCLP$(\pi^N)$ and mdCLP$^*(\pi^N)$ the following bounds holds:
\[
V(\mbox{M-CLP}^*) - V(\mbox{M-CLP}) \le 
V(\mbox{dCLP}_2(\pi^N)) - V(\mbox{dCLP}_1(\pi^N))
\le \Upsilon(N) \epsilon 
\]
where
\[
\Upsilon(N) = c\Tt \sum_{i=1}^N\Delta U^{*i} - b\Tt\sum_{i=1}^N \Delta P^{*i} > 0
\]
\end{proposition}

{\em Proof}.
 The first inequality follows from Proposition \ref{thm.dclp-opt}. To evaluate the second inequality we note that the optimal solutions of mdCLP$(\pi^N)$  and  mdCLP$^*(\pi^N)$ are feasible but suboptimal solutions of  dCLP$_1(\pi^N)$ and dCLP$_2(\pi^N)$.  Calculating the objective values of dCLP$_1(\pi^N)$, dCLP$_2(\pi^N)$ for the solutions $\bu^{0*},  \Delta U^*,$ $ \bu^{N*}$,  $\bp^{N*},  \Delta P^*, \bp^{0*}$  we have:
\begin{eqnarray}
\label{eqn.dopt1}
& V(dCLP_1(\pi^N)) \ge  (\gamma + c T)\Tt  \bu^{0*} + \big(\hgamma +\hc_1 \big)\Tt \Delta U^* + \gamma\Tt \bu^{N*}, \nonumber \\
& V(dCLP_2(\pi^N)) \le (\beta + b T)\Tt  \bp^{N*} + \big(\hbeta + \hb_2 \big)\Tt\Delta P^*  + \beta\Tt \bp^{0*}
\end{eqnarray}

On the other hand, because mdCLP$(\pi^N)$  and mdCLP$^*(\pi^N)$ are dual problems:
\begin{eqnarray}
\label{eqn.dopt2}
 & V(\text{mdCLP}(\pi^N)) = (\gamma + c T)\Tt  \bu^{0*} + \big(\hgamma +\hc_2 \big)\Tt \Delta U^* + \gamma\Tt \bu^{N*} = \nonumber \\
 & = (\beta + b T)\Tt  \bp^{N*} + \big(\hbeta + \hb_1 \big)\Tt\Delta P^*  + \beta\Tt \bp^{0*}  = V(\text{mdCLP}^*(\pi^N)) 
\end{eqnarray}

Combining (\ref{eqn.dopt1}) and (\ref{eqn.dopt2}), after easy manipulations we get:
\begin{equation*}
 V(dCLP_2(\pi^N)) - V(dCLP_1(\pi^N)) \le  \big(\hb_2 \Tt - \hb_1 \Tt\big)  \Delta P^*  + \big(\hc_2 \Tt - \hc_1 \Tt\big) \Delta U^* = \epsilon \Upsilon(N) \qquad\endproof
\end{equation*}

\begin{proposition}
\label{thm.finlim}
The sequence of optimal values of the dual problems mdCLP$(\pi^N)$ and mdCLP$^*(\pi^N)$ has finite lower and upper bounds, $V_L,\,V_U$. 
\end{proposition}

{\em Proof}.
We consider the  single interval  partition $\pi^1$, where we have the problem:
\begin{eqnarray}
\label{eqn.pLP1}
 &\max &z = (\gamma + c T)\Tt  \bu + (\gamma + c T)\Tt U + \gamma\Tt \bu^T    \nonumber
 \hspace{0.6in}   \\
 &\mbox{s.t.} &  A \bu + x^0 = \beta  \nonumber \\
 \mbox{mdCLP}(\pi^1)  &\mbox{} &  A \bu + A U + x^t = \beta + b T   \\
 &\mbox{} &  A \bu + A U + A \bu^T + x^T = \beta + b T  \nonumber \\
&& \quad \bu,\, \bu^T, U,\,x^0, x^t, x^T \ge 0.  \nonumber
\end{eqnarray}
An optimal solution to (\ref{eqn.pLP1}) can be extended to a feasible solution of mdCLP$(\pi^N)$ as follows: $\displaystyle \bu^0 = \bu,\; \bu^N = \bu^T,\;
\Delta U=\left[\frac{U}{N},\dots,\frac{U}{N}\right],\;  \hx=\left[ \left(1 - 2 \epsilon\right) x^0 + 2\epsilon x^t,\dots, 2 \epsilon x^0 +  \left(1 - 2 \epsilon\right) x^t,\, x^t \right],\;$ $ x^0=x^0,\; \bx^N = x^T$.
 Hence the following inequality holds:
\begin{eqnarray*}
 V( \mbox{mdCLP}(\pi^N)) &\ge& (\gamma+cT)\Tt \bu+\gamma\Tt U+\hc_2\Tt \Delta U+\gamma\Tt \bu^T  \\ 
& =&  (\gamma+cT)\Tt\bu+\gamma\Tt U+c\Tt \Big(\frac{T}{2}+\epsilon \Big)U+\gamma\Tt \bu^T \\
&\ge &  (\gamma+cT)\Tt\bu+\gamma\Tt U+\Big( c^+ \frac{T}{2}  - c^- T \Big)\Tt U+\gamma\Tt \bu^T
= V_L
\end{eqnarray*}
where $c^+_j=\max(c_j,0),\,c^-_j=\max(-c_j,0)$, and we recall that $\epsilon \le \frac{T}{2}$.

Similarly, by considering the dual, an upper bound is obtained in terms of the solution of the dual test problem:
\[
V_U = (\beta+b T)\Tt\bp+\beta\Tt P+\Big(b^+ T - b^- \frac{T}{2} \Big)\Tt P+\beta\Tt \bp^T \qquad\endproof
\]


\subsection{Bounding the discrete solutions}
In this section we assume that M-CLP as well as M-CLP$^*$ satisfy the Slater type condition \ref{def.slater}.
Under this assumption we will show that all the optimal solutions of mdCLP$(\pi^N)$ and mdCLP$^*(\pi^N)$ are uniformly bounded.

We consider first the  sequence of primal problems $\{\text{mdCLP}(\pi^N)\}_{N=1}^\infty$.  
 We use the following notations:
\begin{eqnarray*}
&& \{\bu^{0\,*(N)}, \Delta U^{*(N)}, \bu^{N\,*(N)}\}_{N=1}^\infty,  \mbox{ are the optimal solutions}, \\
&&  u^{*(N)}(t_i)=\Delta U^{i\,*(N)}/2 \epsilon,   \qquad i=1,\ldots,N, \\
&&  u^{*(N)}(t) \text{ is the piecewise constant extension of the } u^{*(N)}(t_1),\ldots,u^{*(N)}(t_N)\\
&&  U^{*(N)}(t) = \bu^{0\,*(N)}+\int_0^t u^{*(N)}(s) ds, \; t\in [0,T), \quad U^{*(N)}(T) = U^{*(N)}(T-) + \bu^{N\,*(N)}.
\end{eqnarray*}

\begin{proposition}
\label{thm.fseqbound}
If  M-CLP$^*$ is strictly feasible then all $J$ elements of $U^{*(N)}(T)$ have a uniform finite upper bound. 
\end{proposition}

{\em Proof}.
Take any $j=1,\ldots,J$, we will show that  $U_j^{*(N)}(T)$ is bounded by a constant $\Psi_j$ for all $N$.  Recall that $U^{*(N)}(t)$ are non decreasing, so this bound will hold for all $U_j^{*(N)}(t),\,t\in[0,T]$.

We choose $N_0$  large enough and corresponding $\epsilon$ small enough so that:
\[ 
\epsilon \le \frac{\alpha_1}{2} \implies \frac{T}{2N_0} \le \frac{\alpha_1}{2} \implies N_0 \ge \frac{T}{\alpha_1} 
\]
where $\alpha_1$ is a small constant, to be determined later.  We will find a uniform bound for 
$U^{*(N)}(T), N\ge N_0$.  

We use the following notation:
 \begin{eqnarray*}
&&
\delta=\left[ \begin{array}{l}\delta_1 \\ \vdots \\\delta_J \end{array} \right] \mbox{where } 
\delta_j =  \left\{ \begin{array}{ll} \frac{\alpha_1}{2}  & \mbox{if } c_j > 0 \\ 
0 & \mbox{if } c_j \le 0 \end{array} \right., \\
&& \hat{c} =\left[ \begin{array}{l} c \\ \vdots \\ c \end{array} \right] \mbox{and } 
\hat{\delta} =\left[ \begin{array}{l}\delta \\ \vdots \\\delta \end{array} \right]
\mbox{are the $N$-fold vectors of $c$'s and $\delta$'s}
\end{eqnarray*}

Consider the following discrete linear optimization problem, for a discrete error bound:
\begin{eqnarray}
\label{eqn.EBLP1pi}
&\Psi^{(1)}_{j}(N) =    \max &  \frac{\alpha_1}{2} \bu_j^0 \; + \; 
\frac{\alpha_1}{2} \sum_{i=1}^N \Delta U_{j}^i \; + \; \frac{\alpha_1}{2} \bu_j^N \nonumber 
\\
 \mbox{dEBLP$(\pi^N)$} & \mbox{s.t.} &
   (\gamma + \delta + c T)\Tt  \bu^0 + (\hgamma + \hat{\delta} + \hc_1  )\Tt \Delta U + (\gamma + \delta)\Tt \bu^N \ge V_L   \hspace{0.4in} \\
 & & \text{Constraints of mdCLP}(\pi^N) \nonumber
\end{eqnarray}
One can see that $\hc_1 + \hat{\delta} \ge \hc_1 + \hat{c} \epsilon = \hc_2$ and hence, by Proposition \ref{thm.finlim} the first constraint of dEBLP$(\pi^N)$ holds for the solution $\{\bu^{0\,*(N)}, \Delta U^{*(N)}, \bu^{N\,*(N)}\}$.  Hence the optimal solution of 
mdCLP$(\pi^N)$ is feasible for dEBLP$(\pi^N)$.  
In particular, it follows that $\Psi^{(1)}_{j}(N) \ge U_j^{*(N)}(T)$

The problem dEBLP$(\pi^N)$ is a discretization of the following continuous linear programming problem:
\begin{eqnarray}
\label{eqn.EBCLP1pi}
&\Psi^{(2)}_{j} =    \max &  \frac{\alpha_1}{2} \int_0^T d  U_{j}(t)dt  \nonumber \\
 \mbox{EBCLP} & \mbox{s.t.} &  \int_0^T \left(\gamma +\delta + (T-t)c \right)\Tt  d U(t)   \ge V_L  \hspace{1.5in}  \\
 & & \text{Constraints of M-CLP} \nonumber
\end{eqnarray}
The continuous linear program EBCLP  is not formulated exactly as an M-CLP problem, 
the difference being that the first constraint has linearly varying coefficients rather than constant coefficients.
Nevertheless, one can show by  similar arguments that propositions \ref{thm.discrete}, \ref{thm.feas-ext}  still hold, and so for every $N$, $U^{*(N)}(t)$ is a feasible solution of EBCLP.  We now have that $\Psi^{(2)}_j \ge  \Psi^{(1)}_j $.

EBCLP is obviously feasible.  We now need to show that it is bounded.  We formulate the following dual problem to EBCLP:
 \begin{eqnarray}
\label{eqn.GEBCLP1pi*}
& \Psi^{(3)}_{j} =  \min &\displaystyle  \int_0^T \left(\beta + (T - t) b \right)\Tt  dP(t)  -  V_L P^O \nonumber   \\
  \mbox{EBCLP$^*$}  &\mbox{s.t.} & 
    A\Tt P(t) \ge  (\gamma +\delta + c t) P^O  +  \frac{\alpha_1}{2} \mathbf{e}^j, \quad 0\le t \le T, \\
&& P^O \ge 0,  P(t)\ge 0, \mbox{ non-decreasing and right continuous on }[0,T].
 \nonumber  
\end{eqnarray} 
where $\be^{j}$ is the $j$th unit vector.

It is straightforward to check that weak duality holds between  problems EBCLP and EBCLP$^*$.
Hence, if EBCLP$^*$ is feasible, we have $\Psi^{(3)}_j \ge \Psi^{(2)}_j $.

It remains to show that EBCLP$^*$ is feasible.  We now use the assumption that M-CLP$^*$ is strictly feasible.  Hence there exists a vector of functions $\tilde{P}(t),\,t\in[0,T]$ that satisfy:
\begin{eqnarray}
\label{eqn.GEBCLP1pi**}
 &&  A\Tt  \tilde{P}(t)  \ge \gamma + c t   \nonumber \\
 &&  A\Tt  \tilde{P}(t)  \ge \gamma + c t + \alpha_1,  \\
&& \quad \tilde{P}(t) \ge 0, \mbox{ non-decreasing and right continuous on }[0,T].
  \nonumber 
\end{eqnarray}
for some small enough value $\alpha_1$.  This gives us our choice for the value of $\alpha_1$. 

It is now easy to check that $P^O=1$ and $\tilde{P}(t),\,t\in [0,T]$ is a feasible solution of EBCLP$^*$, indeed, for $P^O=1$:
\[
A\Tt  \tilde{P}(t)  \ge \gamma + c t + \alpha_1 \ge 
(\gamma +\delta + c t) P^O  +  \frac{\alpha_1}{2} \mathbf{e}^j, \quad 0\le t \le T.
\]
Let $\Psi_j^{(4)}$ be the value of the objective of EBCLP$^*$ for this solution.  We have:
$\Psi^{(4)}_j \ge \Psi^{(3)}_j \ge \Psi^{(2)}_j \ge  \Psi^{(1)}_j \ge  U_j^{*(N)}(T),\,N \ge N_0$.

Finally:
\[
\Psi_j = \max\{\Psi^{(4)}_j, U_j^{* (N)}(T),\,N=1,\ldots,N_0\} \ge U_j^{* (N)}(T) \mbox{ for all $N$}. \qquad\endproof
\]

Let $P^{*(N)}(t)$ be defined for the optimal solutions of mdCLP$^*(\pi^N)$, similar to $U^{*(N)}$.
A similar proof shows that if M-CLP is strictly feasible, we can construct for any $k=1,\ldots,K$ a bound: $\Phi_k \ge P_k^{*(N)}(T)$.


\subsection{Strong duality}
\hspace*{1 cm}

\begin{theorem}
\label{thm.strongduality}
If M-CLP and M-CLP$^*$ are strictly feasible, then both have optimal solutions, and there is no duality gap.
\end{theorem}
\begin{proof}
We show first that there is no duality gap.  In Proposition \ref{the.gapbound} we have seen that 
\[
V(\mbox{M-CLP}^*) - V(\mbox{M-CLP}) \le 
V(\mbox{dCLP}_2(\pi^N)) - V(\mbox{dCLP}_1(\pi^N))
\le \Upsilon(N) \epsilon 
\]
where
\[
\Upsilon(N) = c\Tt U^{*(N)}(T) - b\Tt P^{*(N)}(T)
\]
In Proposition \ref{thm.fseqbound} we saw that all components of $U^{*(N)}(T),\,P^{*(N)}(T)$ are uniformly bounded,   by quantities $\Psi_j,\,\Phi_k$.  We therefore have a uniform bound $\Upsilon$:
\[
\Upsilon(N) \le \Upsilon = \sum_{j=1}^J c_j^+ \Psi_j +  \sum_{k=1}^K b_k^- \Phi_k,
\]
where  $c^+_j=\max(c_j,0),\,b^-_k=\max(-b_k,0)$. Hence, 
\[
0\le V(\mbox{M-CLP}^*) - V(\mbox{M-CLP}) \le  \epsilon \Upsilon
\]
and letting $N\to \infty$, so that $\epsilon \to 0$, we get $V(\mbox{M-CLP}^*) = V(\mbox{M-CLP}) $.

We next show that optimal solutions exist.  We saw that $U^{*(N)}(t)$ are feasible solutions for M-CLP for all $N$.  $U^{*(N)}(t)$ are vectors of non-negative non-decreasing functions, and by Proposition \ref{thm.fseqbound} they are all uniformly bounded. 
By Helly's selection principle (Theorem 5, p. 372 in \cite{kolmogorov-fomin:75}), 
it is then possible to find a subsequence $N_m$ such that $U_j^{*(N_m)}(t)$ converge  pointwise for every $t$ to a non-negative non-decreasing right continuous function $U_j(t),\,t\in [0,T]$, for all $j=1,\ldots,J$.  It is immediate to see that $U(t)$ is a feasible solution for M-CLP. 

 By Helly's convergence theorem (Theorem 4, p. 370 in \cite{kolmogorov-fomin:75}) by the continuity of 
 $\gamma\Tt + c\Tt(T-t)$
\begin{equation} 
\label{eqn.helly2} 
\lim_{N_m \to \infty} \int_0^T (\gamma + (T-t)c)\Tt dU^{(N_m)}(t) = 
\int_0^T (\gamma + (T-t)c)\Tt dU(t),
\end{equation}  
but this limit equals $ V(\text{M-CLP})$, hence $U(t)$ is an optimal solution of M-CLP.
Similarly the dual problem M-CLP$^*$ has an optimal solution.
\qquad\end{proof}


\section{Form of optimal solution}
\label{sec.solform}��
We now consider problems M-CLP that  have an optimal primal solution $U_O(t)$. In particular, this is true if M-CLP are primal and dual strictly feasible (see Theorem \ref{thm.strongduality}). In this section we investigate properties of the optimal solution. 

\begin{proposition}
\label{thm.mclpcont}
$c\Tt U_{O}(t)$ is continuous on $(0, T)$.
\end{proposition} 
\begin{proof}
Assume the contrary. Then, because  $c\Tt U_{O}(t)$ is of bounded variations, it has only jump discontinuities, with left and right limits.  Consider a `jump' point $t_c$ with  $c\Tt U_{O}^\uparrow = c\Tt U_{O}(t_c +) - c\Tt U_{O}(t_c -) \ne 0$. 

Assume first that $c\Tt U_{O}^\uparrow > 0$. 
Let $t_a < t_c$ be a point such that $| c\Tt U_O(t_c-) - c\Tt U_O(t)| < \frac{1}{4}c\Tt U_{O}^\uparrow$ for all $t \in [t_a, t_c)$.  Such a point exists because $U_O$ has left and right limits at $t_c$.   
 Consider the following solution of M-CLP:
\[
\tilde{U}(t)=\left\{ \begin{array}{ll} 
\displaystyle U_{O}(t_a) + \frac{t-t_a}{t_c-t_a} \left(U_{O}(t_c+) - U_{O}(t_a) \right), & t \in [ t_a, t_c ], \\
U_{O}(t), & t \notin [ t_a, t_c].  \\
\end{array} \right.
\]
It is clear that $\tilde{U}(t)$ is feasible. Comparing the objective values for $U_{O}(t)$ and $\tilde{U}(t)$ we obtain:

\begin{eqnarray*}
&& \int_0^T (\gamma\Tt + c\Tt(T-t))d\tU(t) - \int_0^T (\gamma\Tt + c\Tt(T-t))dU_O(t) \\
&& = \int_{t_a}^{t_c+} c\Tt t dU_O(t) -  \int_{t_a}^{t_c+} c\Tt t d\tU(t)  \\
&&=  \frac{t_c-t_a}{2}  \Big(c\Tt U_O^\uparrow + c\Tt U_O(t_c-)+ c\Tt U_O(t_a)\Big) 
-   \int_{t_a}^{t_c-} c\Tt U_O(t) dt  \\
&&\ge  \frac{t_c-t_a}{2} \Big(c\Tt U_O^\uparrow -  \big(2 \sup_{t\in [t_a,t_c)}c\Tt {U}_O(t)- c\Tt {U}_O(t_a) - c\Tt U_O(t_c-)\big)\Big) > 0 
\end{eqnarray*}
where in the second equality we replace Lebesgue-Stieltjes integral by Riemann-Stieltjes integral and integrate by parts.
 This contradicts the optimality of $U_{O}(t)$.  
A similar contradiction is obtained if $c\Tt U_{O}^\uparrow < 0$,  considering a point $t_a > t_c$.
We conclude that $c\Tt U_{O}(t)$ has no jumps, and hence is continuous on $(0,T)$.
\qquad\end{proof}

\begin{proposition}
\label{thm.mclpconc}
$c\Tt U_{O}(t)$ is concave on $(0, T)$.
\end{proposition} 
\begin{proof}
Assume the contrary. Then, since by Proposition \ref{thm.mclpcont} $c\Tt U_{O}(t)$ is continuous, there exists an interval $(t_1, t_2)$ such that:
\[ c\Tt U_{O}(t) <  c\Tt U_{O}(t_1) + \frac{t-t_1}{t_2-t_1} \left( c\Tt U_{O}(t_2) - c\Tt U_{O}(t_1) \right), \quad t \in (t_1, t_2)
\]
Consider the following solution of M-CLP:
\[
U^*(t)=\left\{ \begin{array}{ll} 
\displaystyle U_{O}(t_1) + \frac{t-t_1}{t_2-t_1} \left(U_{O}(t_2) - U_{O}(t_1) \right), & t \in ( t_1, t_2 ), \\
U_{O}(t), & t \notin ( t_1, t_2).  \\
\end{array} \right.
\]
It is clear that $U^*(t)$ is feasible. We note also that from our assumption it follows that $c\Tt U_{O}(t) < c\Tt U^*(t)$ on $(t_1, t_2)$. Comparing objective values for $U_{O}(t)$ and $U^*(t)$ we obtain, similar to the proof of Proposition \ref{thm.mclpcont}:
\begin{eqnarray*}
&& \int_0^T (\gamma\Tt + c\Tt(T-t))dU^*(t) - \int_0^T (\gamma\Tt + c\Tt(T-t))dU_O(t) \\
&& = \int_{t_1}^{t_2} c\Tt t dU_O(t) - \int_{t_1}^{t_2} c\Tt t dU^*(t)   \\
&& = \int_{t_1}^{t_2} c\Tt U^*(t) dt - \int_{t_1}^{t_2} c\Tt U_O(t) d t  
> 0.
\end{eqnarray*}
 This contradicts the optimality of $U_{O}(t)$. Hence $c\Tt U_{O}(t)$ is concave.
\qquad\end{proof}

By the Lebesgue decomposition theorem any feasible solution of M-CLP can be represented as $U(t)= U_a(t) + U_s(t)$, where $U_a(t)$ is an absolutely continuous function and $U_s(t)$ is a singular function, including a discrete singular (`jump') part and a continuous singular part.   
\begin{proposition}
\label{thm.cudecr}
Consider an optimal solution $U_O(t)$ and its Lebesgue decomposition $U_O(t)= U_a(t) + U_s(t)$, and let  $u(t)= \frac{dU_a(t)}{dt}$.  Then the following holds:
\begin{description}
 \item[(i)]
\[
 \frac{d}{dt} c\Tt U_O(t)= c\Tt u(t)
 \]
 \item[(ii)]
 $c\Tt u(t)$ is a non-increasing function.
\item[(iii)]
\[  
\int_{0-}^T (\gamma + (T-t)c)\Tt dU_O(t) = \gamma\Tt U_O(T) + c\Tt T U_O(T-) - \int_0^T c\Tt t u(t) dt
\]
\end{description}
\end{proposition}
{\em Proof}.
(i)  By Proposition \ref{thm.mclpconc} $c\Tt U_O(t)$ is concave and hence it  is absolutely continuous on the interval $(0,T)$. Therefore, by the uniqueness of the Lebesgue decomposition, $c\Tt (U_O(t)-U_O(0)) = c\Tt U_a(t)$ on this interval.  

(ii) That $c\Tt u(t)$ is non-increasing follows from the concavity of $c\Tt U_O(t)$.

(iii)  Because $(\gamma + (T-t)c)\Tt$ is continuous the Lebesgue-Stieltjes integral above can be replaced by the Riemann-Stieltjes integral.  
\begin{eqnarray*}  
&& \int_{0-}^T (\gamma+ (T-t)c)\Tt dU_O(t) \\
&& = \gamma\Tt U_O(T) + \int_{0-}^T (T-t) d c\Tt U_O(t) \\
&& = \gamma\Tt U_O(T) + c\Tt T U_O(0) + \int_{0+}^{T-} (T-t) d c\Tt U_O(t) \\
&& = \gamma\Tt U_O(T) + c\Tt T U_O(0) + \int_0^T (T-t) c\Tt u(t) dt \\
&& = \gamma\Tt U_O(T) + c\Tt T U_O(T-) - \int_0^T c\Tt t u(t) dt \qquad\endproof
\end{eqnarray*}  

For  part (iii) of the next theorem we need the following non-degeneracy assumption:
\begin{assumption}
\label{asm.nondeg}
The vector $c$ is in general position to
the matrix $\left[ A\Tt \; I  \right]$
(it is not a linear combination of any less than $J$ columns).
\end{assumption}
\begin{theorem}
\label{thm.solform}
Assume that  M-CLP/M-CLP$^*$ have optimal solutions $U_O(t), P_O(t)$ with no duality gap, then: 

\noindent
(i)  $c\Tt U_O(t)$ is piecewise linear on $(0, T)$ with a finite number of breakpoints. 

\noindent

(ii) There exists an optimal solution $U^*(t)$ that is continuous and piecewise linear on $(0, T)$.

\noindent
(iii) Under the non-degeneracy assumption  \ref{asm.nondeg}, every optimal solution is of this form, and furthermore, $U_O(t)-U_O(0+)$ is unique over $(0,T)$.
\end{theorem}
\begin{proof}
(i) By the Lebesgue differentiation theorem $U_O(t), P_O(t)$ can be differentiated at least almost everywhere. Let $S$ be the set where $U_O(t)$ is not differentiable and $S^*$ be the set where $P_O(T-t)$ is not differentiable. Let $E_0 = [0,T] \setminus S \cup S^*$. Then the complementary slackness condition (\ref{eqn.gcompslack}) can be rewritten as:
\[ 
\int_{E_0} q(T-t) \Tt u(t) dt = \int_{S \cup S^*} q(T-t)\Tt dU_s(t) = \int_{E_0} x(T-t)\Tt p(t) dt =  \int_{S \cup S^*} q(T-t)\Tt dP_s(t) = 0
\]
where $u(t) = \frac{dU_O(t)}{dt}, p(t) = \frac{dP_O(t)}{dt}$ are the densities of $U_O(t), P_O(t)$ on $E$ and $x(t)=\beta + bt - A U_O(t), q(t) = A\Tt P_O(t) - \gamma - ct$ are slack functions.   Hence, we must have 
for every point of $E_0$ apart from another  set of measure zero $S_1$, that  $q(T-t)\Tt u(t) = x(T-t)\Tt p(t) = 0$.
Let $E = E_0 \setminus S_1$.  

At the same time, differentiating the constraints of M-CLP/M-CLP$^*$ everywhere on the set $E$ we obtain:
\begin{equation} 
\label{eqn.rates_constraints}
Au(t) + \dx(t) = b \qquad A\Tt p(t) - \dq(t) = c
\end{equation}
where $\dx(t), \dq(t)$ are the slopes of the corresponding slacks. 

Note that on $E$, $x(t) = 0 \implies \dx(t) = 0$ and $q(t) = 0 \implies \dq(t) = 0$, by non-negativity. Hence, for every $t \in E$ the following holds:
\begin{equation}
\label{eqn.rates_complslack}
\dq(T-t)\Tt u(t) = \dx(T-t)\Tt p(t) = 0
\end{equation}

Recall also that $U_O,P_O$ and hence also $U_a,P_a$ are non-decreasing, so:
\begin{equation}
\label{eqn.unonneg}
u(t) \ge0, \qquad p(t)\ge 0.
\end{equation}

Consider now any point $t\in E$, and the values of $u(t),x(t),p(T-t),q(T-t)$.  Let $\J(t)$, $\K(t)$ be the indices of the non-zero components of $u(t)$ and of $p(T-t)$ respectively.  
One can see that (\ref{eqn.rates_constraints}), (\ref{eqn.rates_complslack}), (\ref{eqn.unonneg}) imply that $u=u(t),\dx=\dx(t),p=p(T-t),\dq=\dq(T-t)$ are optimal solutions of the following pair of linear programming problems:
\begin{equation}
\label{eqn.rates}
\begin{array}{c}
\begin{array}{rcl}
 & \max &  c\Tt u  \\
 &\mbox{s.t.}& A u + I \dx = b \\
\mbox{Rates-LP$(t)$} && u_j \in \setZ \text{ for } j \notin \J(t)\; u_j \in \setP \text{ for } j \in \J(t) \\
&& \dx_k \in \setU \text{ for } k \notin \K(t) \; \dx_k \in \setP \text{ for } k \in \K(t) 
\end{array} \\ 
\\
\begin{array}{rcl}
 & \min &  b\Tt p  \\
&\mbox{s.t.}& A\Tt p - I \dq = c \\
\mbox{Rates-LP$^*(t)$} && p_k \in \setZ \text{ for } k \notin \K(t)\; p_k \in \setP \text{ for } k \in \K(t) \\
&& \dq_j \in \setU \text{ for } j \notin \J(t) \; \dq_j \in \setP \text{ for } j \in \J(t)
\end{array}
\end{array}
\end{equation}
where by $\setZ, \setP, \setU$ we denote the following sign restrictions: $\setZ = \{0\}$ is zero, $\setP = \setR_+$ is non-negative and $\setU = \setR$ is unrestricted.

Let $M$ be the finite number of subsets of indices $\J(t),\K(t)$ for which the dual pair of linear programs (\ref{eqn.rates}) is feasible.  Then it follows that the values of  $c\Tt u(t)$ for all $t\in E$ must be the objective values of an optimal solution  of (\ref{eqn.rates}), for one of these  subsets.  Since by Proposition \ref{thm.cudecr}(ii) $c\Tt u(t)$ is non-decreasing there must exist a partition $0=t_0 < t_1 < \cdots < t_N=T,\;N \le M$ such that $c\Tt u(t)$ is constant over each interval $(t_{n-1},t_n)\cap E$.  Recall that by Proposition \ref{thm.mclpconc} $c\Tt U_O(t)$ is absolutely continuous on $(0,T)$, and hence
 $c\Tt U_O(t) = c\Tt U_O(0) + \int_0^t c\Tt u(t) dt$.  It follows that $c\Tt U_O(t)$ is continuous piecewise linear on $(0,T)$.

(ii) Consider the following solution of M-CLP:
\[
U^*(t)=\left\{ \begin{array}{ll} 
\displaystyle U_{O}(t_n) + \frac{t-t_n}{t_{n+1}-t_n} \left(U_{O}(t_{n+1}) - U_{O}(t_n) \right), & \begin{array}{l}t \in ( t_n, t_{n+1}], \; n=0,\dots, N-2 \\ \text{and }t \in (t_{N-1}, t_N),\end{array} \\
U_{O}(t), & \; t = 0, T.  \\
\end{array} \right.
\]
where $0=t_0 < t_1 < \dots < t_N=T$ is the time partition defined in the proof of (i). Note, that $U^*(t)$ is piecewise linear and absolutely continuous on $(0,T)$. One can see that  $U^*(t)$ is also a feasible solution for M-CLP. Let $u^*(t) = \frac{dU^*(t)}{dt}$. Similar to  Proposition \ref{thm.cudecr}(iii) and based on its proof, we could rewrite the objective value obtained with $U^*(t)$ as follows:
\begin{equation}
\label{eqn.another_obj}
\int_{0-}^T (\gamma+ (T-t)c)\Tt dU^*(t) = \gamma\Tt U_O(T) + c\Tt T U_O(T-) - \int_0^T c\Tt t u^*(t) dt
\end{equation}
 Recall also, that by Proposition \ref{thm.mclpconc} $c\Tt U_O(t)$ is concave on $(0,T)$ and hence absolutely continuous on $(0,T)$. Then, comparing the objective values for $U_{O}(t)$ and $U^*(t)$ we obtain:
\begin{eqnarray*}
&& \int_{0-}^T (\gamma\Tt + c\Tt(T-t))dU_O(t) - \int_{0-}^T (\gamma\Tt + c\Tt(T-t))dU^*(t)   \\ 
&& =  \int_0^T c\Tt t u^*(t) dt - \int_0^T c\Tt t u(t) dt \\
&&=  \sum_{n=0}^{N-1} \int_{t_n}^{t_{n+1}} t \left(\frac{d}{dt} \left(c\Tt U_O(t_n) + \frac{t-t_n}{t_{n+1}-t_n} \left(c\Tt U_{O}(t_{n+1}) - c\Tt U_{O}(t_n)\right)\right) - c\Tt u(t) \right) dt = 0
\end{eqnarray*}
where the first equality follows from Proposition \ref{thm.cudecr}(iii) and (\ref{eqn.another_obj}), and the second follows from (i) of this Theorem. Hence, $U^*(t)$ is an optimal solution of M-CLP.

(iii) Let $U_O(t)$ be any optimal solution of M-CLP.  
As shown in (i), there is a time partition $0=t_0 < t_1 < \dots < t_N=T$ so that $c\Tt U_O(t)$ is continuous piecewise linear, with  constant slope in each interval, where the slopes in successive intervals are strictly decreasing.   Let $I_n=(t_{n-1},t_n],\,n=1,\ldots,N-1,\;I_N=(t_{N-1},t_N)$.
As shown in (ii) we can construct a dual optimal solution $P^*(t)$  which is continuous piecewise linear on $(0,T)$, with breakpoints  $0=t_0 < t_1 < \dots < t_N=T$.   For this dual solution we have 
that $p(t) = \frac{d P^*(t)}{dt}$ is constant on each interval $I_n$,  let $p_n=p(t),\,t\in I_n$ denote this  vector value.   Then as shown in (i), $p_n$ is a solution of the dual Rates-LP$^*$ (\ref{eqn.rates}).  

Consider now $u(t)=\frac{dU_a(t)}{dt}$ which is defined almost everywhere on $(0,T)$.  As shown in (i), $u(t)$ is an optimal solution of the primal Rates-LP 
(\ref{eqn.rates}), and $u(t)$ is complementary slack to $p(t)$ almost everywhere.  By the non-degeneracy assumption \ref{asm.nondeg}, $p_n$ is non-degenerate.  Hence, $u(t)$ is uniquely determined almost everywhere on $I_n$, as the unique solution which is complementary slack to $p_n$.  Denote this solution by $u_n,\,n=1,\ldots,N$.  This uniquely determines $U_a(t),\,t\in (0,T)$, the absolutely continuous part of $U_O(t)$, as the continuous piecewise linear vector of functions with slopes $u_n$ in each interval. 

It remains to show that $U_O$ is absolutely continuous in $(0,T)$, i.e that $U_s(T-)-U_s(0+)=0$.

Assume to the contrary that  in some interval $(t_{m-1}, t_{m}]$ we have $U_s(t_m)-U_s(t_{m-1})>0$ (or if $m=N$, $U_s(T-)-U_s(t_{N-1})>0$).
Define 
\[
U^*_m(t)=\left\{ \begin{array}{ll} 
\displaystyle U^*(t) & t \in I_m\\
U_{O}(t), & t \notin I_m,  
\end{array} \right.   \qquad
u^*_m(t)=\left\{ \begin{array}{ll} 
\displaystyle \frac{dU^*(t)}{dt} & t \in I_m\\
u(t), & t \notin I_m.  
\end{array} \right.
\]
where $U^*(t)$ on $I_m$ is the linear interpolation as defined in proof of (ii),  and $u^*_m=u^*_m(t)$ for $t\in I_m$ is the constant slope $\frac{U_O(t_m)-U_O(t_{m-1})}{t_m-t_{m-1}}$ on $I_m$.   

Similar to (ii), it follows that  $U^*_m(t)$ is also an optimal solution of M-CLP.   Furthermore, since 
the solutions are identical on $t\not\in I_m$, we must by (i) have that $c\Tt u_m = c\Tt u^*_m$.

We now have, on the one hand, that:
\[ u^*_m = \frac{U_{O}(t_m) - U_{O}(t_{m-1})}{t_m-t_{m-1}}  = u_m + \frac{1}{t_m-t_{m-1}} \left(U_{s}(t_m) - U_{s}(t_{m-1})\right) \ne  u_m 
\]

On the other hand, as we saw before, $u^*_m(t)$ must be complementary slack to $p(t)$ almost everywhere, and hence, $u^*_m=u_m$.    This contradiction proves that $U_s(T-)-U_s(0+)=0$, and shows that $U_O(t)$ is absolutely continuous on $(0,T)$.

Furthermore, $U_O(t)$ for $t\in (0,T)$ is continuous piecewise linear, and $U_O(t)-U_O(0+)$ is uniquely determined by the partition $0=t_0 < t_1 < \dots < t_N=T$ and the slope vectors $u_n$.  This completes  the proof or (iii).
\qquad\end{proof}

\subsubsection*{Completion of the proof of Theorem \ref{thm.generalization}} 
(iii) We first show that for objective values $V$ holds $V(SCLP) \le V(\mbox{M-CLP})$. Consider following CLP problem:
\begin{eqnarray*}
\label{eqn.DCLP}
&\min & \int_0^T (\alpha + (T-t)a)\Tt d P(t) + \int_0^T b\Tt q(t) \,dt     \nonumber  \hspace{0.9in} \\
\mbox{DCLP}  \quad&\mbox{s.t.} &  G\Tt\, P(t) + H\Tt q(t) \ge \gamma + c t \\
 && \quad\; F\Tt P(t) + P_s(t) \ge d t \nonumber \\
 && \quad\; - F\Tt P(t) - P_s(t) \ge  - d t \nonumber \\
&& \quad q(t)\ge 0, \quad P(t), P_s(t) \mbox{ non-decreasing and right continuous on [0,T].}   \nonumber
\end{eqnarray*}
which is a generalization of Pullans' dual for the case when $a(t),c(t)$ are affine, and  $b(t)$ is constant. 

One can see that weak duality holds between SCLP (\ref{eqn.PWSCLP})  and DCLP.  One can also see that weak duality holds between  the  M-CLP extension and DCLP.  But under the Slater type condition, the M-CLP/M-CLP$^*$ extensions possess  primal and dual optimal solutions $\tU(t), \tP(t)$ with no duality gap (see Theorem \ref{thm.strongduality}). Now, setting $P(t) = \tP_*(t),\,P_s(t) = \tP_s(t),\, q(t) = \tP^+(t) - \tP^-(t)$ we obtain a feasible solution of DCLP with the same objective value. This solution is optimal for DCLP by weak duality between M-CLP extension and DCLP. Then, by weak duality between SCLP and DCLP $V(SCLP) \le V(DCLP) = V(\mbox{M-CLP})$.  

Now, consider $u^*(t), x^*(t)$ be a feasible solution of SCLP, where $x^*(t)$ is of bounded variation. Because $G \int_0^t u^*(s)ds$ and right hand side of SCLP are both absolutely continuous such solution could be easily found.  Moreover, this solution could be translated to a solution of M-CLP extension as shown in the proof of (i). Consider also an additional constraint $u(t) \le W, 0 \le t \le T$, where $W \ge \max_{j, 0 < t < T} u^*_j(t)$, which preserves SCLP feasibility. We denote SCLP with this additional constraint as SCLP$(W)$. It is clear that $u^*(t), x^*(t)$ still be feasible for SCLP$(W)$. Let M-CLP$(W)$ be an extension of the SCLP$(W)$. One can see that M-CLP$(W)$ is nothing also as M-CLP extension of the SCLP with following additional constraints:
\begin{equation}
\label{eqn.ubconstr} 
\begin{array}{c}
U_*(t) + U_s(t) \le W t \\
- U_*(t) - U_s(t) \le - W t
\end{array}
\end{equation}
It is clear that M-CLP$(W)$ is feasible. Moreover, one could choose $W$ big enough to preserve the  Slater type condition for the M-CLP$(W)$. Furthermore, one can see that the dual M-CLP$(W)^*$ is a relaxation of the  M-CLP$^*$, and therefore the Slater type condition still holds for the  M-CLP$(W)^*$. 

Now, consider $\tU(t)$ to be an optimal solution of M-CLP$(W)$ (existence of the such solution follow from Theorem \ref{thm.strongduality}). One could see, that constraint (\ref{eqn.ubconstr}) require that for this solution $\tbu^0_* = \tbu^N_* = \tbu^0_s = \tbu^N_s = 0$, and hence this solution could be translated back to an optimal solution of SCLP$(W)$ by taking:  
\[ u^{**} = \frac{d \tU_*}{dt}  \quad x^{**} = U^{+}(t) - U^{-}(t)
\]

Finally, consider a sequence $W^{(n)}= nW, n=1,\dots$ and let $\tU^{(n)}$ be a sequence of optimal solutions of M-CLP$(W^{(n)})$, and $u^{(n)}, x^{(n)}$ be a sequence of corresponding optimal solutions of SCLP$(W^{(n)})$. It is clear that feasible region growth in $n$, and hence sequences of objective values $V(SCLP(W^{(n)})) = V(\mbox{M-CLP}(W^{(n)}))$ involving by corresponding solutions are increasing. Moreover, $\tU^{(n)}$ are vectors of non-negative non-decreasing uniformly bounded functions, which are feasible solution of M-CLP.  Hence, letting $n \to \infty$ and repeating arguments from the proof of existing optimal solution (second part of the proof of the Theorem \ref{thm.strongduality}) we obtain: 
\[ 
\lim_{n \to \infty} V(SCLP(W^{(n)})) =  \lim_{n \to \infty} V(\mbox{M-CLP}(W^{(n)})) = V(\mbox{M-CLP})
\]  
 $ \square$


\begin{thebibliography}{20}

\bibitem{anderson:78} {\sc Anderson, E.~J.}
{\em A continuous model for job-shop scheduling.}
Ph.D. Thesis, University of Cambridge, Cambridge, 1978

\bibitem{anderson:81} {\sc Anderson, E.~J. }
{\em A new continuous model for job-shop scheduling.}
International J. Systems Science, {\bf 12}, pp. ~1469--1475, 1981.

\bibitem{anderson-nash:87}
{\sc Anderson, E.~J. and Nash, P. }
{\em Linear Programming in Infinite Dimensional Spaces}.
Wiley-Interscience, Chichester, 1987.

\bibitem{anderson-philpott:89}
{\sc Anderson, E.~J., Philpott, A.~B.} 
{\em A continuous time network simplex algorithm.}
Networks, {\bf 19}, pp. ~395--425, 1989

\bibitem{anderson-philpott:89:2}
{\sc Anderson, E.~J., Philpott, A.~B.}
{\em Erratum, a continuous time network simplex algorithm.}
Networks, {\bf 19}, pp. ~823--827, 1989 

\bibitem{anderson-pullan:96}
{\sc Anderson, E.~J., Pullan, M.~C.}
{\em Purification for separated continuous linear programs.}
Math. Methods Oper. Res. {\bf 43}, pp. ~9--33, 1996 

\bibitem{barvinok:02}
{\sc Barvinok, A.}
{\em A Course in Convexity}.
American Mathematical Society, 2012.

\bibitem{bellman:53}
{\sc Bellman, R.} 
{\em Bottleneck problems and dynamic programming.}
Proc. National Academy of Science {\bf 39}, pp. ~947--951, 1953.

\bibitem{bellman:57}
{\sc Bellman, R.}
{\em Dynamic Programming}.
Princeton University Press, Princeton NJ, 1657.

\bibitem{dantzig:51}
{\sc Dantzig, G.}
{\em Application of the simplex method to a transportation problem.}
in T. Koopmans, ed., ‘Activity Analysis of Production and Allocation’
John Wiley and Sons, New York, pp. ~359--373, pp. ~330--335, 1951. 

\bibitem{grinold:70}
{\sc  Grinold, R.~C.} 
{\em Symmetric duality for continuous linear programs.}
SIAM J. Applied Mathematics, {\bf 18}, pp. ~32--51, 1970.

\bibitem{kolmogorov-fomin:75}
{\sc  Kolmogorov, A.~A.~N., Fomin, S.~V.}
{\em Introductory real analysis.}
Dover, 1975

\bibitem{levinson:66}
{\sc Levinson, N.}
{\em A class of continuous linear programming problems.}
J. Mathematical Analysis Applications, {\bf 16}, pp. ~73--83, 1966.

\bibitem{pullan:93}
{\sc Pullan, M.~C.} 
{\em An algorithm for a class of continuous linear programs.}
SIAM J. Control and Optimization, {\bf 31}, pp. ~1558--1577, 1993.

\bibitem{pullan:95}
{\sc Pullan, M.~C.}
{\em Forms of optimal solutions for separated continuous linear programs.}
SIAM J. Control and Optimization, {\bf 33}, pp. ~1952--1977, 1995.

\bibitem{pullan:96}
{\sc Pullan, M.~C.}
{\em A duality theory for separated continuous linear programs.}
SIAM J. Control and Optimization, {\bf 34}, pp. ~931--965, 1996.

\bibitem{pullan:97}
{\sc Pullan, M.~C.}
{\em Existence and duality theory for separated continuous linear programs.} 
Math. Model. Syst., {\bf 3}, pp. ~219--245, 1997

\bibitem{pullan:00}
{\sc Pullan, M.~C.}
{\em  Convergence of a general class of algorithms for separated continuous linear programs.}
SIAM J. Control and Optimization, {\bf 10}, pp. ~722--731, 2000.

\bibitem{pullan:02}
{\sc Pullan, M.~C.}
{\em An extended algorithm for separated continuous linear programs.}
Math. Program., {\bf 93}, pp. ~415--451, 2002.

\bibitem{shapiro:01}
{\sc Shapiro, A.}
{\em On duality theory of conic linear problems.}
In: Goberna, M.~A., Lopez, M.~A. (eds.) Semi-Infinite Programming, Chap. 7, pp. ~135--165,
Kluwer, Netherlands, 2001

\bibitem{tyndall:65}
{\sc Tyndall, W.~F.}
{\em A duality theorem for a class of continuous linear programming problems.}
SIAM J. Applied Mathematics, {\bf 13}, pp. ~644--666, 1965

\bibitem{tyndall:67}
{\sc Tyndall, W.~F.}
{\em An extended duality theorem for continuous linear programming problems.}
SIAM J. Applied Mathematics, {\bf 15}, pp. ~1294--1298, 1967.

\bibitem{wang-zang-yao:09}
{\sc Wang, X., Zhang, S., Yao, D.}
{\em Separated Continuous Conic Programming: Strong Duality and an Approximation Algorithm.}
SIAM J. Control and Optimization, {\bf 48}, pp. ~2118--2138, 2009.

\bibitem{weiss:08}
{\sc Weiss, G.}
{\em A simplex based algorithm to solve separated continuous linear programs.}
Mathematical Programming Series A, pp. ~151--198, 2008.
\end{thebibliography}
\end{document}